\documentclass[11pt]{amsart}

\usepackage{hyperref}
\usepackage[usenames]{color}
\usepackage{amsmath,amsthm,amsfonts,amssymb}

\newtheorem{theorem}{Theorem}[section]
\newtheorem{lemma}[theorem]{Lemma}

\newtheorem{definition}[theorem]{Definition}

\theoremstyle{definition}
\newtheorem{example}[theorem]{Example}

\theoremstyle{remark}

\numberwithin{equation}{section}

\begin{document}

\title{Asymptotic density of rational sets in free abelian groups}

\author{Anton Menshov}
\address{Institute of Mathematics and Information Technologies\\Omsk State Dostoevskii University}
\curraddr{}
\email{menshov.a.v@gmail.com}
\thanks{}


\date{}

\begin{abstract}
In this paper we study asymptotic density of rational sets in free abelian group $\mathbb{Z}^n$ of rank $n$.
We show that any rational set $R$ in  $\mathbb{Z}^n$ has asymptotic density.
If $R$ is given by its semi-simple decomposition we show how to compute its asymptotic density.
\end{abstract}

\maketitle

\section{Preliminaries}
\label{sec:prelim}

\subsection{Rational sets}
\label{subsec:rational}

Let $M$ be a monoid, the class of {\em rational} subsets of $M$ is the least class $\mathfrak{R}$ of subsets of $M$ satisfying the following conditions:
\begin{enumerate}
\item[(1R)] $\emptyset, \{m\} \in \mathfrak{R}$ for $m \in M$.
\item[(2R)] If $X, Y \in \mathfrak{R}$ then $X \cup Y \in \mathfrak{R}$.
\item[(3R)] If $X, Y \in \mathfrak{R}$ then $XY = \{xy \mid x \in X,~ y \in Y\} \in \mathfrak{R}$.
\item[(4R)] If $X \in \mathfrak{R}$ then $X^* = \bigcup\limits_{n \geq 0} X^n \in \mathfrak{R}$.
\end{enumerate}

If $M$ is a free finitely generated monoid, then according to Kleene's theorem the rational sets are precisely the subsets of $M$ recognizable by finite state automata.

We may also define the smaller class of {\em unambiguously rational} subsets of $M$ by leaving condition (1R) but replacing conditions (2R)-(4R) by stronger conditions (2UR)-(4UR) as follows:
\begin{enumerate}
\item[(2UR)] If $X, Y \in \mathfrak{R}$ and $X \cap Y = \emptyset$ then $X \cup Y \in \mathfrak{R}$.
\item[(3UR)] If $X, Y \in \mathfrak{R}$ and the product $XY$ is unambiguous (i.~e., $x_1y_1 = x_2y_2$ for $x_1,x_2 \in X, y_1,y_2 \in Y$ implies $x_1=x_2,y_1=y_2$), then $XY \in \mathfrak{R}$.
\item[(4UR)] If $X \in \mathfrak{R}$ and $X$ is the basis of free submonoid $X^*$ of $M$, then $X^* \in \mathfrak{R}$.
\end{enumerate}

We will study rational sets in commutative monoids, so we will use additive notation.
In line with this in conditions (3R) and (3UR) $XY$ will be replaced with $X+Y$.

The study of rational sets in a commutative monoid $M$ is simplified by the following notions.
A subset
\[
X = a + B^*
\]
with $a \in M, B \subset M, B$ finite, is called {\em linear}.
If $B^*$ is a free commutative monoid with basis $B$, then $X$ is called {\em simple}.
If $B = \{ b_1, \dots, b_r \}$ is a set of $r$ elements, then every element $x \in X$ may be written as
\[
x = a + n_1 b_1 + \dots + n_r b_r,
\]
where $n_i \in \mathbb{N}$.
If $X$ is simple, then $n_1, \dots, n_r$ are unique.

A finite union of linear sets is called {\em semi-linear}.
A finite {\em disjoint} union of simple sets is called {\em semi-simple}.

Cleary every semi-linear set is rational and every semi-simple set is unambiguously rational. The converse is also true (see \cite{ES}).

It is known that in a free monoid $M$ every rational set is unambiguously rational.
The main result of \cite{ES} states that in a commutative monoid $M$ every rational set is unambiguously rational, and it follows from \cite{ES} that every rational set $R \subseteq M$ can be presented as a semi-simple set.

\subsection{Lattice point counting}
\label{subsec:lattice_points}
Counting lattice points in the integral dilates of a subset of Euclidean space $\mathbb{R}^n$ is a well known problem.
For rational polytopes this problem has been studied in the $1960$s by the French mathematician Eug\`{e}ne Ehrhart (see \cite{BR}).
We recall some basic notions here.
A {\em convex polytope} in $\mathbb{R}^n$ is a finite intersection of closed half-spaces, i.e.,
\[
\mathcal{P} = \{ \mathbf{x} \in \mathbb{R}^n \mid A\mathbf{x} \leq b \}, \ \ where \ A \in \mathbb{R}^{mn}, b \in \mathbb{R}^m.
\]
A bounded convex polytope $\mathcal{P}$ is called {\em rational} if all of its vertices have rational coordinates.
We will call the least common multiple of the denominators of the coordinates of the vertices of $\mathcal{P}$ the {\em denominator} of $\mathcal{P}$.

We recall that a {\em quasipolynomial} $Q$ is an expression of the form $Q(t) = c_n(t)t^n + \dots + c_1(t)t + c_0(t)$, where $c_0, \dots, c_n$ are periodic functions in $t$ and $c_n$ is not the zero function.
The {\em degree} of $Q$ is $n$, and the least common period of $c_0, \dots, c_n$ is the {\em period} of $Q$.

For $t \in \mathbb{Z}^{+}$ and $S \subseteq \mathbb{R}^n$ denote $t S = \{ t\mathbf{x} \mid \mathbf{x} \in S \}$ the $t^{th}$ dilate of $S$.
We denote the {\em lattice-point enumerator} for the $t^{th}$ dilates of $S$ by
\[
L_{S}(t) = |t S \cap \mathbb{Z}^n|.
\]
The {\em dimension} of $S \subseteq \mathbb{R}^n $ is the dimension of the affine space
\[
\mathrm{span}~S = \{ \mathbf{x} + \lambda (\mathbf{y}-\mathbf{x}) \mid \mathbf{x}, \mathbf{y} \in S, \lambda \in \mathbb{R} \}
\]
spanned by $S$.
If a polytope $\mathcal{P}$ has dimension $d$, we call $\mathcal{P}$ a $d$-polytope.

\begin{theorem}[\cite{BR}, Theorem 3.23]\label{th:ehrhart}
If $\mathcal{P}$ is a rational convex $n$-polytope, then $L_{\mathcal{P}}(t)$ is a quasipolynomial in $t$ of degree $n$.
Its period divides the denominator of $\mathcal{P}$.
\end{theorem}

This result is due to Eug\`{e}ne Ehrhart, in whose honor $L_{\mathcal{P}}$ is called the {\em Ehrhart quasipolynomial} of $\mathcal{P}$.

The leading coefficient of $L_{\mathcal{P}}(t)$ is equal to $n$-dimensional volume of $\mathcal{P}$, i.e., it is a constant.

We note that there is an algorithm by Alexander Barvinok to compute Ehrhart quasipolynomials. Barvinok's algorithm is polynomial in fixed dimension, it has been implemented in the software package $\mathsf{LattE}$ \cite{L}.

Let $\mathbf{w_1}, \dots, \mathbf{w_n}$ be lineary independent vectors in $\mathbb{R}^n$.
The set
\[
\Lambda = \Lambda(\mathbf{w_1}, \dots, \mathbf{w_n}) =
\{ \alpha_1 \mathbf{w_1} + \dots + \alpha_n \mathbf{w_n} \mid \alpha_i \in \mathbb{Z} \}
\]
is called the {\em lattice} with basis $\{\mathbf{w_1}, \dots, \mathbf{w_n}\}$.
The number
\[
d(\Lambda) = | \det (\mathbf{w_1}, \dots, \mathbf{w_n}) |
\]
is called the {\em determinant} of the lattice.

We note that theorem \ref{th:ehrhart} remains true if we replace the standard integer lattice $\mathbb{Z}^n$ by an arbitrary lattice $\Lambda(\mathbf{w_1}, \dots, \mathbf{w_n})$.
Indeed, consider the matrix $A = (\mathbf{w_1}, \dots, \mathbf{w_n})$ formed by $\mathbf{w_i}$ as columns and let $\psi$ be the linear transformation corresponding to $A$.
Then $\Lambda = \psi(\mathbb{Z}^n)$ and $\mathbb{Z}^n = \psi^{-1}(\Lambda)$.
If $\mathcal{P}$ is a rational $n$-polytope with respect to the basis $\{\mathbf{w_1}, \dots, \mathbf{w_n}\}$, then $\psi^{-1}(\mathcal{P})$ is a rational $n$-polytope with respect to the standard basis of $\mathbb{R}^n$ and
\[
    L_{\mathcal{P}, \Lambda}(t) =
    \left| \{ t\mathcal{P} \cap \Lambda \} \right| =
    \left| \{ t\psi^{-1}(\mathcal{P}) \cap \mathbb{Z}^n \} \right|.
\]
Given that $\mathrm{vol}(\mathcal{P}) = |\det(A)| \mathrm{vol}(\psi^{-1}(\mathcal{P}))$ we get that the leading coefficient of $L_{\mathcal{P}, \Lambda}(t)$ is equal to $\frac{\mathrm{vol}(\mathcal{P})}{d(\Lambda)}$.

It is known that if $S \subset \mathbb{R}^n$ is a bounded convex $n$-dimensional set with piecewise smooth boundary then its volume can be computed as
\[
\mathrm{vol}(S) =
\lim_{r \to \infty} \frac{|\mathbb{Z}^n \cap rS|}{r^n}.
\]
Hence $|\mathbb{Z}^n \cap rS| \sim \mathrm{vol}(S) r^n$.
Arguing as above one can show that for an arbitrary lattice $\Lambda = \Lambda(\mathbf{w_1}, \dots, \mathbf{w_n})$
\begin{equation}\label{eq:points-asymptotics}
|\Lambda \cap rS| \sim \frac{\mathrm{vol}(S)}{d(\Lambda)} r^n.
\end{equation}
If $S$ has dimension $d < n$ then $|\Lambda \cap rS| = o(r^n)$.

\subsection{Asymptotic density}
\label{subsec:density}
A {\em stratification} of a countable set $T$ is a  sequence $\{T_r\}_{r \in \mathbb{N}}$ of non-empty finite subsets $T_r$ whose union is $T$. 
Stratifications are often specified by length functions. A {\em length function} on $T$ is a map $l: T \to \mathbb{N}$ from $T$ to the nonnegative integers $\mathbb{N}$ such that the inverse image of every integer is finite. The corresponding spherical and ball stratifications are formed by {\em spheres} $S_r = \{x \in T \mid l(x) = r\}$ and {\em balls} $B_r = \{x \in T \mid l(x) \leq r\}$.

\begin{definition}
The asymptotic density of $M \subset T$ with respect to a stratification $\{T_r\}$ is defined to be
\[
\rho(M) = \lim_{r\to\infty} \rho_r(M), \ \ \ \mbox{where} \ \ \  \rho_r(M) = \frac{|M \cap T_r|} {|T_r|}
\]
 when the limit exists. Otherwise, we use the limits 
\[
\bar{\rho}(M) = \limsup_{r\to\infty} \rho_r(M), \ \ \ \underline{\rho}(M) = \liminf_{r\to\infty} \rho_r(M),
\]
and call them upper and lower asymptotic densities respectively.
\end{definition}

Asymptotic density is one of the tools for measuring sets in infinite groups (see \cite{BMS} for details).

Let $\mathbb{Z}^n$ be a free abelian group of rank $n$.
We identify $\mathbb{Z}^n$ with the standard integer lattice in Euclidean space $\mathbb{R}^n$.
We assume that $\mathbb{R}^n$ is equipped with $\| \cdot \|_p$-norm for some $1 \leq p \leq \infty$.
This norm induces the lenght function $ l_p: \mathbb{Z}^n \to \mathbb{N}$ with balls
\[
B_{p,r}(\mathbb{Z}) = B_{p,r} \cap \mathbb{Z}^n,
\]
where $B_{p,r} = \{ v \in \mathbb{R}^n \mid \|v\|_p \leq r \}$ is the ball of radius $r$ in $\mathbb{R}^n$ with respect to $\| \cdot \|_p$.
For $M \subseteq \mathbb{Z}^n$ the corresponding relative frequences are defined by
\[
\rho_{p,r}(M) = \frac{|M \cap B_{p,r}|}{|B_{p,r}(\mathbb{Z})|}.
\]
We will denote the corresponding densities by $\rho_p$, $\bar{\rho}_p$ and $\underline{\rho}_p$.

From (\ref{eq:points-asymptotics}) it follows that $|B_{p,r}(\mathbb{Z})| \sim \mathrm{vol}(B_{p,1}) r^n $.
Thus
\begin{equation}\label{eq:balls-limit}
\lim_{r \to \infty} \frac{|B_{p,r+1}(\mathbb{Z})|}{|B_{p,r}(\mathbb{Z})|} = 1.
\end{equation}
We will call $\rho_p$ {\em invariant} if for any $M \subseteq \mathbb{Z}^n$ with $\rho_p(M) = \rho$ and any $v \in \mathbb{Z}^n$ we have $\rho_p(v+M) = \rho$.
According to \cite[Proposition 2.3]{BV} natural density in finitely generated groups is invariant.
In particular, $\bar{\rho}_1$ is invariant.
A similar arguments show that $\rho_p$ is invariant.

\begin{lemma}\label{lm:shift}
Asymptotic density $\rho_p$ is invariant for any $1 \leq p \leq \infty$.
\end{lemma}
\begin{proof}
It is enough to show that $\rho_{p}(e + M) = \rho_{p}(M)$, where $e = \pm e_i $ and $\{e_1, \dots, e_n\}$ is the standard basis of $\mathbb{R}^n$.

Observe, that $e + (M \cap B_{p,r}) \subseteq (e + M) \cap B_{p,r+1}$.
Indeed, if $w = e + m$ where $m \in (M \cap B_{p,r})$ then $\|w\|_p \leq \|e\|_p + \|m\|_p \leq r+1$, so $w \in B_{p,r+1}$.
Hence $|M \cap B_{p,r}| \leq |(e+M) \cap B_{p,r+1} |$.

Also observe, that $(e+M) \cap B_{p,r} \subseteq e + (M \cap B_{p,r+1})$.
Indeed, if $w = e + m \in (e+M) \cap B_{p,r}$ then $\|m\|_p = \|w-e\|_p \leq r+1$, so $m \in B_{p,r+1}$.
Hence $|(e+M) \cap B_{p,r}| \leq |M \cap B_{p,r+1}|$.

Combining all the above for $r > 1$ we get
\begin{equation}\label{eq:lemma1-bounds}
\frac{|M \cap B_{p,r-1}|}{|B_{p,r}(\mathbb{Z})|} \leq
\rho_{p,r}(e+M) \leq
\frac{|M \cap B_{p,r+1}|}{|B_{p,r}(\mathbb{Z})|}.
\end{equation}
From (\ref{eq:balls-limit}) it follows that
\[
\lim_{r \to \infty} \frac{|M \cap B_{p,r-1}|}{|B_{p,r}(\mathbb{Z})|} =
\lim_{r \to \infty} \frac{|M \cap B_{p,r+1}|}{|B_{p,r}(\mathbb{Z})|} =
\rho_p(M)
\]
hence
\[
\lim_{r \to \infty} \rho_{p,r}(e+M) = \rho_p(M).
\]
\end{proof}

\section{Main results}
\label{sec:main}
Subgroups in $\mathbb{Z}^n$ are particular type of rational sets.
According to \cite[Proposition 2.4]{BV} in a finitely generated group $G$ natural density of a subgroup $H$ of finite index is equal to $\frac{1}{[G:H]}$.
In particular, for a subgroup $H \leq \mathbb{Z}^n$ of finite index $\rho_1(H) =\frac{1}{[\mathbb{Z}^n:H]}$.
Observe, that $H$ is a lattice with determinant $d(H) = [\mathbb{Z}^n : H]$, hence by (\ref{eq:points-asymptotics})
\[
\rho_p(H) =
\lim_{r \to \infty} \frac{|H \cap B_{p,r}|}{|B_{p,r}(\mathbb{Z})|} = \frac{1}{[\mathbb{Z}^n : H]}.
\]
Clearly, if $H$ is of infinite index then $\rho_p(H) = 0$.

It follows from \cite{ES} that any rational set $R$ in $\mathbb{Z}^n$ can be presented as a {\em semi-simple} set
\begin{equation}\label{eq:semi-simple}
    R = \bigcup_{i = 1}^{k} (a_i + B_i^*).
\end{equation}
If $\rho_p(B_i^*)$ exists for any $B_i^*$ then by lemma \ref{lm:shift} and since the union (\ref{eq:semi-simple}) is disjoint
\begin{equation}\label{eq:density-sum}
    \rho_p(R) = \sum_{i = 1}^{k} \rho_p(a_i + B_i^*) =
    \sum_{i = 1}^{k} \rho_p(B_i^*).
\end{equation}
Thus, to compute $\rho_p(R)$ it suffices to compute asymptotic density for free commutative monoids in $\mathbb{Z}^n$.

We recall that a {\em cone} $\mathcal{K} \subseteq \mathbb{R}^n$ generated by $\mathbf{w_1}, \dots, \mathbf{w_m}$ is a set of the form
\[
\mathcal{K} =
\mathrm{cone}(\mathbf{w_1}, \dots, \mathbf{w_m}) = 
\{ \alpha_1 \mathbf{w_1} + \dots + \alpha_m \mathbf{w_m} \mid \alpha_i \geq 0, \alpha_i \in \mathbb{R} \}.
\]

\begin{lemma}\label{lm:monoid-density}
Let $B^* \subset \mathbb{Z}^n$ be a free commutative monoid with basis $\{b_1, \dots, b_k\}$, then $\rho_p(B^*)$ exists and $\rho_p(B^*) > 0$ if and only if $k = n$.
\end{lemma}
\begin{proof}
Since $\{b_1, \dots, b_k\}$ is the basis of $B^*$, if follows that $\{b_1, \dots, b_k\}$ are lineary independent in $\mathbb{R}^n$, thus $k \leq n$ and we can consider the lattice $\Lambda = \Lambda(b_1, \dots, b_k)$.
Denote $\mathcal{P} = B_{p,1} \cap \mathrm{cone}(b_1, \dots, b_k)$.
Observe, that
\[
B^* \cap B_{p,r} = B^* \cap (B_{p,r} \cap \mathrm{cone}(b_1, \dots, b_k)) = \Lambda \cap r \mathcal{P}.
\]
If $p=1$ or $p=\infty$ then $\mathcal{P}$ is a rational $k$-polytope and $|\Lambda \cap r\mathcal{P}| = L_{\mathcal{P}, \Lambda}(r)$ is its Ehrhart quasipolynomial.
Next
\[
\rho_{p,r}(B^*) =
\frac{|\Lambda \cap r\mathcal{P}|}{|B_{p,r}(\mathbb{Z})|}
\]
which implies that $\rho_p(B^*) = \lim\limits_{r \to \infty} \rho_{p,r}(B^*) = 0$ if $k < n$, and if $k = n$
\begin{equation}\label{eq:fcm-density}
\rho_p(B^*) =
\lim_{r \to \infty} \rho_{p,r}(B^*) = \frac{\mathrm{vol}(\mathcal{P})}{\mathrm{vol}(B_{p,1}) d(\Lambda)}.
\end{equation}
\end{proof}

Notice, that (\ref{eq:fcm-density}) gives us the method for computing asymptotic density of free commutative monoids.

The following theorem immediately follows from (\ref{eq:density-sum}) and lemma \ref{lm:monoid-density}.

\begin{theorem}
For any rational subset $R \subseteq \mathbb{Z}^n$ $\rho_p(R)$ exists.
If $R$ is given by its semi-simple decomposition
$
R = \bigcup\limits_{i = 1}^{k} (a_i + B_i^*)
$
then $\rho_p(R) = \sum\limits_{i = 1}^{k} \rho_p(B_i^*)$.
\end{theorem}

\begin{example}
We will compute asymptotic density of the free commutative monoid $ M = \{ (2, 1),~ (1, 2) \}^{*} \subset \mathbb{Z}^2$.
Denote $ v_{1} = (2, 1), \; v_{2} = (1, 2) $ and consider the lattice $ \Lambda = \Lambda(v_{1}, v_{2}) $ with determinant  $ d(\Lambda) = 3 $.
By (\ref{eq:fcm-density}) we get
\begin{align*}
    \rho_1(M) &= 
    \frac{\mathrm{vol}(B_{1,1} \cap \mathrm{cone}(v_{1}, v_{2}))}{\mathrm{vol}(B_{1,1}) \cdot d(\Lambda)} =
    \frac{\frac{1}{6}}{2 \cdot 3} =
    \frac{1}{36}, \\
    \rho_{\infty}(M) &= 
    \frac{\mathrm{vol}(B_{\infty,1} \cap \mathrm{cone}(v_{1}, v_{2}))}{\mathrm{vol}(B_{\infty,1}) \cdot d(\Lambda)} =
    \frac{\frac{1}{2}}{4 \cdot 3} =
    \frac{1}{24}.
\end{align*}
Thus for different $\| \cdot \|_p$-norms the corresponding asymptotic densities are not necessarily equal.
\end{example}


\end{document}